\documentclass[11pt,reqno]{amsart}
\usepackage[utf8]{inputenc}
\usepackage[T1]{fontenc}
\usepackage{lmodern}
\usepackage[english]{babel}
\usepackage{amsmath,a4wide}
\usepackage{xfrac}
\usepackage{esint}
\usepackage{stmaryrd,mathrsfs,bm,amsthm,mathtools,yfonts,amssymb,color,braket,booktabs,graphicx,graphics,amsfonts}
\newcommand{\R}{\mathbb{R}}

\newcommand{\N}{\mathbb{N}}

\theoremstyle{plain}
\newtheorem{theorem}{Theorem}[section]
\newtheorem*{theorem*}{Theorem}
\newtheorem{lemma}[theorem]{Lemma}
\newtheorem*{lemma*}{Lemma}
\newtheorem{prop}[theorem]{Proposition}
\newtheorem*{prop*}{Proposition}

\newtheorem*{corollary*}{Corollary}
\theoremstyle{definition}

\newtheorem*{example*}{e.g.}
\newtheorem{remark}[theorem]{Remark}
\newtheorem*{remark*}{Remark}
\newtheorem{assumption}[theorem]{Assumption}
\newtheorem*{assumption*}{Assumption}
\numberwithin{equation}{section}
\newtheorem{definition}[theorem]{Definition}
\newtheorem*{definition*}{Definition}

\title{Time global existence of generalized BV flow\\
via the Allen--Cahn equation}
\author{Kiichi Tashiro}
\address{(K.Tashiro) Department of Mathematics, Tokyo Institute of Technology, 2-12-1 Ookayama, Meguroku, Tokyo 152-8551, Japan}
\email{tashiro.k.ai@m.titech.ac.jp}

\begin{document}
	
	\maketitle
	\markboth{Kiichi Tashiro}{Time global existence of generalized BV flow via the Allen--Cahn equation}
	
	\begin{abstract}
		We show that a mean curvature flow obtained as the limit of the Allen--Cahn equation is not only a Brakke flow but also a generalized BV flow proposed by Stuvard and Tonegawa. 
	\end{abstract}
	
	\section{Introduction}
	\label{Intro}
        One of the most important geometric flows, the mean curvature flow (henceforth referred to as MCF), has been studied in the mathematical literature since the 1970's. 
        The unknown of MCF is a one-parameter family $ \{ M_t \}_{ t \geq 0 } $ of surfaces in the Euclidean space (or more generally some Riemannian manifold) such that the normal velocity vector $ v $ of $ M_t $ equals its mean curvature vector $ h $ at each point for every time, i.e., $ v = h $ on $ M_t $. Given a compact smooth surface, a unique smooth solution exists until singularities such as shrinkage and neck pinching occur. To consider the solutions that allow singularities, various frameworks of weak solutions of the MCF have been proposed: we mention the Brakke flow \cite{brakke1978motion}, the level set solution \cite{evans1991motion,chen1991uniqueness}, the BV flow \cite{luckhaus1995implicit}, the $ L^2 $ flow \cite{roger2008allen}, the De Giorgi type flow \cite{hensel2021new} and the generalized BV flow \cite{StuvardTonegawa+2022}.
	\vskip.5\baselineskip
        The Allen--Cahn equation 
        is the following simple reaction-diffusion
        semilinear PDE that can produce a MCF in the singular perturbation limit: 
        \begin{equation*}
        \tag{AC}\label{AC0}
        \begin{cases}
            \ \partial_t \varphi^{ \varepsilon } = \Delta \varphi^{ \varepsilon } - \frac{ W' ( \varphi^{ \varepsilon } ) } { \varepsilon^2 } & \text{on } \R^n \times ( 0 , \infty ) ,\\
            \ \varphi^{ \varepsilon } ( \cdot , 0 ) = \varphi^{ \varepsilon }_0 ( \cdot ) & \text{on } \R^n .
        \end{cases}
    \end{equation*}
        Here $ \varepsilon > 0 $ is a small parameter and $ W $ is a double-well potential with local minima at $\pm 1$, for example, $ W( s ) := ( 1 - s^2 )^2 / 2 $. We are interested in the characterization of the limit problem as $ \varepsilon \rightarrow 0 $, where one expects that $ \varphi^{ \varepsilon } \approx \pm 1$ for a bulk region, and the transition layer $ \{ \varphi^{ \varepsilon } \approx 0 \} $ moves by the mean curvature. It is known in the most general setting of geometric measure theory that the limit energy concentration measure $\mu_t$ is a Brakke flow \cite{ilmanen1993convergence,tonegawa2003integrality}. The purpose of this paper is to show that the MCF arising from the Allen--Cahn equation is a generalized BV flow in addition to being a Brakke flow. It is natural to consider the relationship between the phase function $ \varphi $ and the Brakke flow $ \mu_t $, and we prove that $ \varphi $ satisfies the BV-type formula. Namely, the main claim is the following, roughly speaking:
	\begin{theorem*}
		Let $ \{ \mu_t \}_{ t \geq 0 } $ be the Brakke flow and $ \varphi ( x , t ) = \chi_{ E_t } ( x ) $ be the phase function obtained as a limit of (AC). Then, for all test function $ \phi \in C^1_c ( \R^n ) $ and $ 0 \leq t_1 < t_2 < \infty $, we have
		\begin{equation}
			\label{BV}
				\left.\int_{ \R^n } \phi ( x )\,\varphi ( x , t )\ d x\,\right|_{ t = t_1 }^{ t_2 }
				= \int_{ t_1 }^{ t_2 } \int_{ \R^n } \phi ( x )\,( h ( x , t )  \cdot \nu ( x , t ) )\ d | \nabla \varphi ( \cdot , t ) | ( x ) d t,
		\end{equation}
		where $ h ( \cdot , t ) $ is the generalized mean curvature vector of $ \mu_t $, $ | \nabla \varphi ( \cdot , t ) | $ is the perimeter measure of the phase function $ \varphi ( \cdot , t ) $ and $ \nu ( \cdot , t ) $ is the unit outer normal vector of $ | \nabla \varphi ( \cdot , t ) | $.
	\end{theorem*}
        Note that (\ref{BV}) gives an explicit formula for the volume change of the phase: assuming that the inital datum $ E_0 $ is bounded, we have
        \begin{equation}
            \mathcal{ L }^n ( E_{ t_2 } ) -
            \mathcal{ L }^n ( E_{ t_1 } ) = \int_{ t_1 }^{ t_2 }
            \int_{ \R^n }( h ( x , t ) \cdot \nu ( x , t ) ) \,d | \nabla \varphi ( \cdot , t ) | ( x ) d t.
        \end{equation}
        For smooth MCFs, the above formula holds naturally while it is not obvious for generalized MCFs. Luckhaus and Sturzenhecker \cite{luckhaus1995implicit} introduced the notion of the BV flow by characterizing the motion law of the boundary using (\ref{BV}) and BV functions. Namely, a family of sets of finite perimeter $ \{ E_t \}_{ t \geq 0 } $ is a BV flow if the perimeter measures $ \{ | \nabla \chi_{ E_t } | \}_{ t \geq 0 } $ have the generalized mean curvature vector $ h ( \cdot , t ) $ satisfying (\ref{BV}). The generalized BV flow proposed by Stuvard and Tonegawa \cite{StuvardTonegawa+2022} is a pair of phase function and Brakke flow, which allows possible higher integer multiplicities ($ \geq 2 $). They proved the existence under a very general setting (even with multi-phase cases). If (\ref{BV}) holds, it is known by the work of Fischer et al. \cite{fischer2020local} that the BV flow is unique until some topological changes occur, thus partially resolving the issue of superfluous non-uniqueness of Brakke flows. There are some existence results of BV flows such as \cite{luckhaus1995implicit,laux2016convergence,laux2018convergence}, but these studies impose a reasonable (but non-trivial) assumption that the approximate solutions converge to the limit without loss of surface energy. In contrast, we prove the existence of (generalized) BV flows without any extra assumptions, with the caveat that the accompanying Brakke flow may possess possible higher integer multiplicities. We mention that the similar conclusion can be derived \cite{tashiro2023existence} for flows obtained by the elliptic regularization \cite{ilmanen1994elliptic}. 
	\vskip.5\baselineskip
        We next give a detailed account of the related works on the MCFs using the Allen--Cahn equation. As a pioneering work, Ilmanen \cite{ilmanen1993convergence} 
        proved that the limit measure of (AC)
        is a rectifiable Brakke flow by using
        Huisken's monotonicity formula \cite{huisken1990}.
        Additionally Tonegawa \cite{tonegawa2003integrality} proved that 
        the limit measure is integral. By modifying the equation (\ref{AC0}), varifold solutions with various additional terms are studied and we mention \cite{takasao2016existence,qi2018convergence,takasao2020existence,jiang2020convergence,takasao2023existence}. As a conditional result, Laux and Simon \cite{laux2018convergence} gave the time-local existence of BV flow with forcing term in the multi-phase case. Hensel and Laux \cite{hensel2021new} proposed the new concept of weak solution, called the De Giorgi type flow, and studied the existence and the weak-strong uniqueness property. By using the relative entropy, the simplified proof of the singular limit of (\ref{AC0}) is given in \cite{fischer2020convergence}
        and the case of coupling with the Navier--Stokes equation \cite{hensel2023sharp} and volume preservation \cite{kroemer2023quantitative}
        have been studied.
    \vskip.5\baselineskip
    The key observation of the present paper is that the two properties, being an $ L^2 $ flow (which follows from being Brakke flow) and the absolute continuity of phase boundary measure with respect to the Brakke flow (both as space-time measures), lead to equality (\ref{BV}). To this end, we find the convergence of the velocity vector representing the motion of phase boundaries using the concept of measure function pairs by Hutchinson \cite{hutchinson1986second}. More precisely, if $ \{ \mu_t \}_{ t \geq 0 } $ is a Brakke flow arising from the Allen--Cahn equation, by the existence of the velocity, we may obtain $ d | ( \nabla , \partial_t ) \varphi | \ll d \mu_t d t $, where $ | ( \nabla , \partial_t ) \varphi | $ is the space-time perimeter measure of $ \varphi $. Once this is done, we may recover the formula (\ref{BV}) using a suitable version of the co-area formula (see \cite[Theorem 13.4]{maggi2012sets} for example) from geometric measure theory. The idea to prove the main results of this paper is 
    similar to \cite{roger2008allen,StuvardTonegawa+2022,tashiro2023existence}, though there are 
    some fine differences. We also point out that 
    similar conclusions can be derived for more 
    general flows as in \cite{jiang2020convergence,qi2018convergence,takasao2016existence,takasao2017existence,takasao2020existence,takasao2023existence}
    by the same strategy.
	\vskip.5\baselineskip
	The paper is organized as follows. In Section \ref{Preliminaries}, we set our notation and explain the Allen--Cahn equation and the main result. In Section \ref{proofofmainresult1}, we show that the absolute continuity between the perimeter measure of the phase function and the Brakke flow, and the boundaries motion by the mean curvature is expressed in the equality (\ref{BV}), and we then prove that the Allen--Cahn equation gives a generalized BV flow in
 the limit.
	
	\section{Preliminaries and Main Results}\label{Preliminaries}
	
	\subsection{Basic Notation}
	We shall use the same notation for the most part adopted in \cite{StuvardTonegawa+2022,tashiro2023existence}.
    \vskip.5\baselineskip
    In particular, the ambient space we will be working in is the Euclidean space $ \R^n $, and $ \R^+ $ will denote the interval $ [ 0 , \infty ) $. The coordinates $ ( x , t ) $ are set in the product space $ \R^n \times \R $, and $ t $ will be thought of and referred to as ``time''. The symbols $ \mathbf{ p } $ and $ \mathbf{ q } $ will denote the projections of $ \R^n \times \R $ onto its factor, so that $ \mathbf{ p } ( x , t ) = x $ and $ \mathbf{ q } ( x , t ) = t $. If $ A \subset \R^n $ is (Borel) measurable, $ \mathcal{L}^n(A) $ will denote the Lebesgue measure of $ A $, whereas $ \mathcal{ H }^k (A) $ denotes the $ k $-dimensional Hausdorff measure of $ A $. When $ x \in \R^n $ and $ r > 0 $, $ B_r ( x ) $ denotes the closed ball centered at $ x $ with radius $ r $. More generally, if $ k $ is an integer, then $ B^k_r ( x ) $ will denote closed balls in $ \R^k $. The symbols $ \nabla , \nabla^{ \prime } , \Delta , \nabla^2 $ denote the spatial gradient and the full gradient in $ \R^n \times \R $, Laplacian, and Hessian, respectively. The symbol $ \partial_t $ will denote the time derivative.
	\vskip.5\baselineskip
	A positive Radon measure $ \mu $ on $ \R^n $ (or ``space-time'' $ \R^{ n } \times \R^+ $) is always regarded as a positive linear functional on the space $ C^0_c ( \R^n) $ of continuous and compactly supported functions, with the pairing denoted by $ \mu ( \phi ) $ for $ \phi \in C^0_c (\R^n) $. The restriction of $ \mu $ to a Borel set $ A $ is denoted $ \mu \llcorner_A $, so that $ ( \mu \llcorner_A ) ( E ) := \mu ( A \cap E ) $ for any Borel sets $ E \subset \R^n $. The support of $ \mu $ is denoted $ \mathrm{ supp }\, \mu $, and it is the closed set defined by
    \[
        \mathrm{ supp }\, \mu := \left\{ x \in \R^n\,\middle|\,\mu ( B_r ( x ) ) > 0 \text{ for all } r > 0\,\right\}.
    \]
    For $ 1 \leq p \leq \infty $, the space of $ p $-integrable functions with respect to $ \mu $ is denoted $ L^p ( \mu ) $. If $\mu=\mathcal L^n$, $ L^p (\mathcal L^n ) $ is simply written $ L^p ( \R^n ) $. For a signed or vector-valued measure $ \mu $, $ | \mu | $ denotes its total variation. For two Radon measures $ \mu $ and $ \overline{ \mu } $, when the measure $ \overline{ \mu } $ is absolutely continuous with respect to $ \mu $, we write $ \overline{ \mu } \ll \mu $.
    \vskip.5\baselineskip
    We say that a function $ f \in L^1 ( \R^n ) $ has a bounded variation, written $ f \in BV ( \R^n ) $, if
    \[
    \sup \left\{ \int_{\R^n} f \,\mathrm{ div } X\,dx\,\middle|\,X \in C^1_c ( \R^n ; \R^n ) ,\,\| X \|_{ C^0 } \leq 1 \right\} < \infty.
    \]
    If $ f \in BV ( \R^n ) $, then there exists an $ \R^n $-valued Radon measure (which we will call the measure derivative of $ f $ denoted by $ \nabla f $) satisfying
    \[
        \int_{\R^n} f \,\mathrm{ div } X\,dx = - \int_{\R^n} X \cdot d \nabla f\,\text{for all}\,X \in C^1_c ( \R^n ; \R^n).
    \]
	For a set $ E \subset \R^n $, $ \chi_E $ is the characteristic function of $ E $, defined by $ \chi_E = 1 $ if $ x \in E $ and $ \chi_E = 0 $ otherwise. We say that $ E $ has a (locally) finite perimeter if $ \chi_E \in BV ( \R^n ) ( \in BV_{loc} ( \R^n ) ) $. When $ E $ is a set of (locally) finite perimeter, then the measure derivative $ \nabla \chi_E $ is the associated Gauss-Green measure, and its total variation $ | \nabla \chi_E | $ is the perimeter measure; by De Giorgi's structure theorem, $ | \nabla \chi_E | = \mathcal{ H }^{ n - 1 } \llcorner_{ \partial^* E } $, where $ \partial^* E $ is the reduced boundary of $ E $, and $ \nabla \chi_E = - \nu_E\,| \nabla \chi_E | = - \nu_E\,\mathcal{ H }^{ n - 1 } \llcorner_{ \partial^* E } $, where $ \nu_E $ is the outer pointing unit normal vector field to $ \partial^* E $. 
	\vskip.5\baselineskip
    A subset $ M \subset \R^n $ is countably $ k $-rectifiable if it admits a covering 
    \[
        M \subset Z \cup \bigcup_{ i \in \N } f_i ( \R^k )
    \]
    where $ \mathcal{ H }^k ( Z ) = 0 $ and $ f_k : \R^k \to \R^n $ is Lipschitz. If $ M $ is countably $ k $-rectifiable, $ \mathcal{ H }^k $-measurable and $ \mathcal H^k ( M ) < \infty $, $ M $ has a measure-theoretic tangent plane called approximate tangent plane for $ \mathcal H^k $-a.e.$ \, x \in M $ (\cite[Theorem 11.6]{simon1983lectures}), denoted by $ T_x\,M $. We may simply refer to it as the tangent plane at $ x \in M $ without fear of confusion. A Radon measure $ \mu $ is said to be $ k $-rectifiable if there are a countably $ k $-rectifiable, $ \mathcal{ H }^k $-measurable set $ M $ and a positive function $ \theta \in L^1_{ loc } ( \mathcal{ H }^k \llcorner_{ M } ) $ such that $ \mu = \theta\,\mathcal{ H }^k \llcorner_{ M } $. This function $ \theta $ is called multiplicity of $ \mu $. The approximate tangent plane of $ M $ in this case (which exists $ \mu $-a.e.) is denoted by $ T_x\,\mu $. When $ \theta $ is an integer for $ \mu $-a.e., $ \mu $ is said to be integral. The first variation $ \delta \mu : C^1_c ( \R^n ; \R^n ) \to \R $ of a rectifiable Radon measure $ \mu $ is defined by
	\[
	\delta \mu ( X ) = \int_{\R^n } \mathrm{ div }_{ T_x\, \mu } X\,d\mu,
	\]
	where $ P_{ T_x \,\mu } $ is the orthogonal projection from $ \R^n $ to $ T_x\, \mu $, and $ \mathrm{ div }_{ T_x\, \mu } X = \mathrm{ tr } ( P_{ T_x\, \mu } \nabla X ) $. For an open set $ U \subset \R^n $, the total variation $ | \delta \mu | ( U ) $ of $ \mu $ is defined by
	\[
	| \delta \mu | ( U ) = \sup \left\{ \delta \mu ( X )\,\middle|\,X \in C^1_c ( U ; \R^n ) ,\,\| X \|_{C^0} \leq 1 \right\}.
	\]
	If the total variation $ | \delta \mu |(\tilde U)$ is finite for any bounded subset $\tilde U$ of $ U $, then $ \delta \mu $ is called locally bounded, and we can regard $ | \delta \mu | $ as a measure.\,If $ | \delta \mu | \ll \mu $, then the Radon--Nikod\'ym derivative (times $-1$) is called
	the generalized mean curvature vector $ h $ of $ \mu $, and we have
	\[
	\delta \mu ( X ) = - \int_{ \R^n } X \cdot h\,d\mu \quad \text{for all } X \in C^1_c ( \R^n; \R^n ).
	\]
	If $ \mu $ is integral, then $ h $ and $ T_x\, \mu $ are orthogonal for $ \mu $-a.e.~by
	Brakke's perpendicularity theorem \cite[Chapter 5]{brakke1978motion}.
	
	\subsection{Weak Notions of Mean Curvature Flow}
	In this subsection, we introduce some weak solutions to the MCF as well as \cite{StuvardTonegawa+2022,tashiro2023existence}. We briefly define and comment upon the three of interest in the present paper: We begin with the notion of Brakke flow introduced by Brakke \cite{brakke1978motion}.
		\begin{definition}
		A family of Radon measures $ \{  \mu_t \}_{ t \in \R^+ } $ in
  $\mathbb R^n$ 
  is \textit{an} ($ n - 1 $)-\textit{dimensional Brakke flow}                        
		if the following four conditions are satisfied:
		\begin{description}
			\item[\quad\textup{(1)}] For a.e. $ t \in \R^+ $, $\mu_t$ is integral and 
			$ \delta \mu_t $ is locally bounded and absolutely continuous with respect to $  \mu_t $ (thus the generalized mean curvature exists for a.e.~$t$, denoted by $h$).
			\item[\quad\textup{(2)}] For all $ s>0 $ and all compact set $ K \subset \R^n$, $ \sup_{ t \in [ 0 , s ] } \mu_t ( K ) < \infty $.
			\item[\quad\textup{(3)}] The generalized mean curvature $ h $ satisfies $ h \in L^2 (d\mu_tdt) $.
			\item[\quad\textup{(4)}] For all $ 0 \leq t_1 < t_2 < \infty $ and all test functions $ \phi \in C^1_c (\R^n \times \R^+ ; \R^+ )$,
			\begin{equation}
				\begin{split}
					\mu_{ t_2 } ( &\phi ( \cdot , t_2 ) ) - \mu_{ t_1 } ( \phi ( \cdot , t_1 ) )\\
					&\leq \int_{ t_1 }^{ t_2 } \int_{ \R^n} \Big( \nabla \phi ( x , t ) - \phi ( x , t )\,h ( x , t ) \Big) \cdot h ( x , t ) + \partial_t \phi ( x , t )\ d \mu_t ( x ) dt.
				\end{split}
				\label{Brakkeineq}
			\end{equation}
		\end{description}
		\label{Brakke}
	\end{definition}
	The inequality (\ref{Brakkeineq}) is motivated by the following identity:
	\begin{equation}
		\left.\int_{ M_t } \phi ( x , t )\,d\mathcal{ H }^{ n - 1 }\right|_{ t = t_1 }^{ t_2 } = \int_{ t_1 }^{ t_2 } \int_{ M_t } ( \nabla \phi - \phi\,h ) \cdot v + \partial_t \phi\,d\mathcal{ H }^{ n - 1 } dt, \label{motivate}
	\end{equation}
	where $ \{ M_t \}_{ t \in [ 0 , T ) } $ is a family of ($ n - 1 $)-dimensional smooth surfaces, $ h ( \cdot , t ) $ is the mean curvature vector of $ M_t $, and $ v $ is the normal velocity vector of $ M_t $. In particular, if $ \{ M_t \}_{ t \in [ 0 , T ) } $ is a smooth MCF (hence $ v = h $), setting $ \mu_t := \mathcal{ H }^{ n - 1 } \llcorner_{ M_t } $ defines a Brakke flow for which (\ref{Brakkeineq}) is satisfied with the equality. Conversely, if $ \mu_t = \mathcal{ H }^{ n - 1 } \llcorner_{ M_t } $ with smooth $ M_t $ satisfies (\ref{Brakkeineq}), then one can prove that $ \{ M_t \}_{ t \in [ 0 , T ) } $ is a classical solution to the MCF. The notion of Brakke flow is equivalently (and originally in \cite{brakke1978motion}) formulated in the framework of varifolds, but we use the above slightly less general formulation using Radon measures, mainly for convenience.
	\vskip.5\baselineskip
	The following definition of $ L^2 $ flow (modified slightly for our purpose) was given by Mugnai and R\"{o}ger \cite{roger2008allen}.	
	\begin{definition}[$ L^2 $ flow]
		A family of Radon measures $ \{ \mu_t \}_{ t \in \R^+ } $ in $ \R^n $ is \textit{an} ($ n - 1 $)-\textit{dimensional $ L^2 $ flow} if it satisfies (1)-(2) in Definition \ref{Brakke} as well as the following:
		\begin{description}
			\item[\textup{(a)}] The generalized mean curvature $ h ( \cdot , t ) $ (which exists for a.e.\,$ t \in \R^+ $ by (1)) satisfies $ h ( \cdot , t ) \in L^2(  \mu_t ; \R^n ) $, and $ d \mu := d \mu_t d t $ is a Radon measure on $ \R^n \times \R^+ $.
			\item[\textup{(b)}] There exist a vector field $ v \in L^2 ( \mu ; \R^n ) $ and a constant $ C = C ( \mu ) > 0 $ such that
			\begin{description}
				\item[\textup{(b'1)}] $ v ( x , t ) \perp T_x \, \mu_t $ for $ \mu $-a.e.\,$ ( x , t ) \in \R^n \times \R^+ $,
				\item[\textup{(b'2)}] For every test functions $ \phi \in C^1_c ( \R^n \times \R^+ ) $, it holds
				\begin{equation}
					\left| \int_0^{ \infty } \int_{ \R^n } \partial_t \phi ( x , t ) + \nabla \phi ( x , t ) \cdot v ( x , t )\ d \mu_t ( x ) dt\,\right| \leq C\,\| \phi \|_{ C^0 }.
					\label{L2flow}
				\end{equation}
			\end{description}
		\end{description}
		\label{L2def}
	\end{definition}
	The vector field $ v $ satisfying (\ref{L2flow}) is called the generalized velocity vector in the sense of $ L^2 $ flow. This definition interprets equality (\ref{motivate}) as a functional expression of the area change.
	\vskip.5\baselineskip
	Finally, we introduce the concept of generalized BV flow suggested by Stuvard and Tonegawa \cite{StuvardTonegawa+2022}.
	\begin{definition}[Generalized BV flow]
		Let $ \{ \mu_t \}_{ t \in \R^+ } $ and $ \{ E_t \}_{ t \in \R^+ } $ be families of Radon measures and sets of finite perimeter, respectively. The pair $ ( \{ \mu_t \}_{ t \in \R^+ } , \{ E_t \}_{ t \in \R^+} ) $ is \textit{a generalized BV flow} if all of the following hold:
		\begin{description}
			\item[\quad\textrm{(i)}] $ \{ \mu_t\}_{ t \in \R^+ } $ is a Brakke flow as in Definition \ref{Brakke}.
			\item[\quad\textrm{(ii)}] For all $ t \in \R^+ $, $ \| \nabla \chi_{ E_t } \| \leq \mu_t $.
			\item[\quad\textrm{(iii)}] For all $ 0 \leq t_1 < t_2 < \infty $ and all test functions $ \phi \in C^1_c ( \R^n \times \R^+ ) $,
			\begin{equation}
				\begin{split}
				\left.\int_{ E_t } \phi ( x , t )\ d x\,\right|^{ t_2 }_{ t = t_1 } &= \int_{ t_1 }^{ t_2 } \int_{ E_t } \partial_t \phi ( x , t )\ d x d t + \int_{ t_1 }^{ t_2 } \int_{ \partial^* E_t } \phi ( x , t )\,( h ( x , t ) \cdot \nu_{ E_t } ( x ) )\ d \mathcal{ H }^{ n - 1 } ( x ) d t.
				\end{split}
                \label{theareachange}
			\end{equation}
		\end{description}
		\label{generalizedBVflow}
	\end{definition}
	If $ \mu_t $ and $ E_t $ satisfy the above definition, we say ``$ v = h $'' in the sense of generalized BV flow. This definition expresses that the interface $ \partial^* E_t $ is driven by the mean curvature of $ \mu_t $. If $ \mu_t = \| \nabla \chi_{ E_t } \| $ for a.e.\,$ t $, the characterization \eqref{theareachange} coincides with the notion of BV flow considered by Luckhaus--Struzenhecker in \cite{luckhaus1995implicit} since the mean curvature of $ \partial^* E_t $ is naturally defined to be $ h ( \cdot, t ) $ in this case. On the other hand, while the original BV flow is characterized only by \eqref{theareachange}, here $ \mu_t $ is additionally a Brakke flow to which one can apply the local regularity theorems \cite{kasai2014general,tonegawa2014second,stuvard2022endtime}.

	
	\subsection{Assumptions and main result}
	\label{mainresults}
        First, we work under the following: 
	\begin{assumption}
		\label{assumption}
            Let $ E_0 \subset \R^n $ be a set of finite perimeter with 
            \begin{equation}
                \mathcal{ L }^n ( E_0 ) + \| \nabla \chi_{ E_0 }\| ( \R^n )< \infty. 
            \end{equation}
        \end{assumption}
        With this $ E_0 $ given, we can always have 
        a sequence of good initial data for (AC):
        \begin{lemma}
            There exist sequences of $ \varepsilon_i \rightarrow 0 $ and $ C^{ \infty } $ functions $ \varphi_0^{ \varepsilon_i }$ such that
            \begin{equation}
		      \begin{split}
                    &- 1 \leq \varphi_0^{ \varepsilon_i } \leq 1, \quad \lim_{ i \rightarrow \infty } \int_{ \R^n } \left| \frac{ \varphi^{ \varepsilon_i }_0 + 1 }{ 2 } - \chi_{ E_0 } \right| d \mathcal{L}^n = 0,\\
                    &\frac{ 1 }{ \sigma } \left( \frac{ \varepsilon_i \, | \nabla \varphi^{ \varepsilon_i }_0 |^2 }{ 2 } + \frac{ W ( \varphi^{ \varepsilon_i }_0 ) }{ \varepsilon_i } \right) d\mathcal{ L }^n \rightharpoonup  \| \nabla \chi_{ E_0 } \|\ \text{as}\ i \to \infty.
		      \end{split}
            \label{apini}
            \end{equation}
        Here $ \sigma = \int_{ - 1 }^{ 1 } \sqrt{ 2 W ( s ) } \, d s $ is the surface energy constant.
        \end{lemma}
    \begin{proof}
    By \cite[Theorem 13.8]{maggi2012sets}, there exists a sequence of open sets $ E_0^i \subset \R^n $ with smooth boundary such that $ \chi_{ E_0^i } \rightarrow \chi_{ E_0 } $ in $ L^1 ( \R^n ) $ and $ \| \nabla \chi_{ E_0^i } \| \rightarrow \| \nabla \chi_{ E_0 } \| $ as measures. Let $ \Psi : \R \rightarrow ( - 1 , 1 ) $ be the unique ODE solution for $ \Psi' = \sqrt{ 2 W ( \Psi ) } $ and $ \Psi ( 0 ) = 0 $. Define $ \varphi_0^{ \varepsilon_i } ( x ) := \Psi ( \tilde d_i ( x ) / \varepsilon_i ) $, where $ \tilde d_i $ is the signed distance function from $ \partial E_0^i $ truncated so that it is smooth on $ \R^n $. With suitably small choice of $ \varepsilon_i $, one can show that all of the properties \eqref{apini} are satisfied for this $ \varphi_0^{ \varepsilon_i } $ (see \cite{modica1987gradient}). 
    \end{proof}
    In fact, the particular form of $ \varphi_0^{ \varepsilon_i } $ described in the proof is not required, and only the properties \eqref{apini} matter in the following. 
    The following is extracted from \cite{ilmanen1993convergence,tonegawa2003integrality}:
    \begin{theorem}
	Let $ \varphi_0^{ \varepsilon_i } $ ( the index $ i $ being omitted in the following for simplicity) be a sequence satisfying \eqref{apini}, and let $\varphi^\varepsilon$ be the solution of (AC). Define a time-parametrized measure
    \begin{equation}
		\mu_t^{ \varepsilon } := \frac{ 1 }{ \sigma } \left( \frac{ \varepsilon\,| \nabla \varphi^{ \varepsilon } ( \cdot , t ) |^2 }{ 2 } + \frac{ W ( \varphi^{ \varepsilon } ( \cdot , t ) ) }{ \varepsilon } \right)\,d\mathcal{ L }^n.
    \end{equation}
    Then there exists a further subsequence (denoted $ \varphi^{ \varepsilon } $) such that
	\begin{description}
		\item[\textup{(a)}] $ \mu_t^{ \varepsilon } \rightharpoonup \mu_t $ for all $ t \in \R^+ $ and $ \{ \mu_t \}_{ t \in \R^+ } $ is a Brakke flow,
		\item[\textup{(b)}] $ ( 1 + \varphi^{ \varepsilon } ( \cdot , t ) ) / 2 \rightarrow \chi_{ E_t } $ in locally $ L^1 ( \R^n ) $ for all $ t \in \R^+ $
        and $ \| \nabla \chi_{ E_t }\| \leq \mu_t $ for all $ t \in \R^+ $.
	\end{description}
	\label{Brakketrans}
	\end{theorem}
    We note that the key element of the proof is the vanishing of the discrepancy measure (see Lemma \ref{discripancy}) which follows from the local point-wise estimate of \cite[Lemma 3.3]{tonegawa2003integrality}. This gives a local Huisken's monotonicity formula and the verbatim proof of \cite{ilmanen1993convergence,tonegawa2003integrality} works to show the claim of Theorem \ref{Brakketrans}. The following claim is the main result of the present paper.
	\begin{theorem}
		\label{mainresult1}
		The pair $ ( \{ \mu_t \}_{ t \in R^+ } , \{ E_t \}_{ t \in \R^+ } ) $ in Theorem \ref{Brakketrans} is a generalized BV flow as described in Definition \ref{generalizedBVflow}.
	\end{theorem}
        
	\section{Proof of Theorem \ref{mainresult1}}\label{proofofmainresult1}
	The pair $ ( \{ \mu_t \}_{ t \in \R^+ } , \{ E_t \}_{ t \in \R^+ } ) $ in Theorem \ref{Brakketrans} satisfies (i) and (ii) of the Definition \ref{generalizedBVflow} due to Theorem \ref{Brakketrans}. Therefore, the goal of this section is to prove the formula (\ref{theareachange}). In the following, let $ E \subset \R^n \times \R^+ $ denote the set
    \begin{equation}\label{defE}
        E := \{ ( x , t ) \mid x \in E_t , t \geq 0 \}
    \end{equation}
    and note that $ ( 1 + \varphi^{ \varepsilon } ) / 2 \rightarrow \chi_E $ locally in $ L^1 ( \R^n \times \R^+ ) $ by the dominated convergence theorem and Theorem \ref{Brakketrans}(b).
 
 \subsection{Absolute continuity of phase boundary measure}
        Even if a family of perimeter measures $ \{ | \nabla \chi_{ E_t } | \}_{ t \geq 0 } $ is a Brakke flow, the pair $ ( \{ | \nabla \chi_{ E_t } | \}_{ t \geq 0 } , \{ E_t \}_{ t \geq 0 } ) $ may not be a generalized BV flow. For example, define
	\[
		E_t = \begin{cases}
			\,\left\{ x \in \R^n \mid | x |^2 \leq 1 - 2 ( n - 1 ) t \right\} &\left( 0 \leq t < \frac{ 1 }{ 4 ( n - 1 ) } \right), \\
			\,\emptyset &\left( \frac{ 1 }{ 4 ( n - 1 ) } \leq t \right),
		\end{cases}
	\]
	this is a simple counterexample, that is, the formula (\ref{theareachange}) fails at $ t = 1 / ( 4 ( n - 1 ) ) $. We can expect such a phenomenon where the formula (\ref{theareachange}) dose not hold to occur due to a discontinuity to time direction in the measure-theoretic sense. In this subsection, in order to ensure that such examples do not occur, we prove $ | \nabla^{ \prime } \chi_E | \ll \mu $.
    \vskip.5\baselineskip
    First, we recall the following upper density bound (\cite[Section 5.1]{ilmanen1993convergence}) which follows from Huisken's monotonicity formula adapted for the Allen--Cahn equation. Note that Ilmanen's result is for a ``well-prepared'' initial data, but the local discrepancy estimate of \cite[Lemma 3.3]{tonegawa2003integrality} allows one to obtain Huisken's monotonicity formula with a small error term on $ [ T , \infty ) $ for any $ T > 0 $ as follows. 
        \begin{lemma}
            \label{denslemma}
            For any $ T > 0 $, there exists $ D = D ( T, \| \nabla \chi_{ E_0 } \| ( \R^n ) , n ) > 0 $ such that
            \[
                \Theta^* ( \mu_t , x ) := \limsup_{ r \to + 0 } \frac{ \mu_t ( B_r ( x ) ) }{ r^{ n - 1 } } \leq D
            \]
            for any $ t \in [ T , \infty ) $.
        \end{lemma}
        Next, we list the following energy bound:
	\begin{lemma}
        \[
            \sup_{ t \in [ 0 , T ] } \mu_t^{ \varepsilon } ( \R^n ) + \frac{ 1 }{ \sigma } \int_0^T \int_{ \R^n } \varepsilon\,( \partial_t \varphi^{ \varepsilon } )^2\ d \mathcal{ L }^n d t \leq \mu_0^{ \varepsilon } ( \R^n )
        \]
        for any $ 0 < \varepsilon < 1 $ and $ 0 < T < \infty $.
		\label{apriori}
	\end{lemma}
        We also quote from \cite{ilmanen1993convergence,tonegawa2003integrality} the vanishing of discrepancy measure:
	\begin{lemma}
		Under the same assumptions of Theorem \ref{Brakketrans}, for any $T>0$, we have
		\begin{equation}
			\int_0^T \int_{ \R^n } \left| \frac{ \varepsilon | \nabla \varphi^{ \varepsilon } |^2 }{ 2 } - \frac{ W ( \varphi^{ \varepsilon } ) }{ \varepsilon } \right| d \mathcal{L}^n  d t=0.
		\end{equation}
		\label{discripancy}
	\end{lemma}
	Finally, we prove the absolute continuity. Mugnai and R\"{o}ger proved the following proposition for the Allen--Cahn functional, and the following proof is almost identical to \cite[Proposition 8.4]{roger2008allen}.
	\begin{prop}
		\label{varphi<<mu}
		For the Radon measures $ \{ \mu_t \}_{ t \in \R^+ } $ and $ \{ E_t \}_{ t \in \R^+ } $ as in Theorem \ref{Brakketrans}, let $ d \mu = d \mu_t dt $ and let $E$ be as in \eqref{defE}. Then we have
		\begin{equation}
			| \nabla^{ \prime } \chi_E | \ll \mu
		\end{equation}
  and recall that $ \nabla^{ \prime } $ is the gradient with respect to the time and space variables. 
	\end{prop}
	
	\begin{proof}
		We define the function by
		\begin{equation}
			w ( r ) := \int_{ - 1 }^{ r } \sqrt{ 2 W ( s ) }\,d s
		\end{equation}
        and recall that $ \sigma = w ( 1 ) $. For a point $ ( x , t ) $ with $ | \nabla \varphi^{ \varepsilon } ( x , t ) | > 0 $, we have
		\begin{equation}
			| \nabla w ( \varphi^{ \varepsilon } ) | = | \nabla \varphi^{ \varepsilon } | \sqrt{ 2 W ( \varphi^{ \varepsilon } ) }\,, \quad | \partial_t w ( \varphi^{ \varepsilon } ) | = | \partial_t \varphi^{ \varepsilon } | \sqrt{ 2 W ( \varphi^{ \varepsilon } ) }
		\end{equation}
		so that
		\begin{equation}
			| \nabla^{ \prime } w ( \varphi^{ \varepsilon } ) | = \sqrt{ 1 + \left( \frac{ \partial_t \varphi^{ \varepsilon } }{ | \nabla \varphi^{ \varepsilon } | } \right)^2 } | \nabla w ( \varphi^{ \varepsilon } ) |. \label{nablaprime}
		\end{equation}
		By Lemma \ref{apriori}, for any $ T > 0 $ we have
		\begin{equation}
            \begin{split}
			&\int_0^T \int_{ \{ | \nabla \varphi^{ \varepsilon } ( \cdot , t ) | > 0 \} } \left( \sqrt{ 1 + \left( \frac{ \partial_t \varphi^{ \varepsilon } }{ | \nabla \varphi^{ \varepsilon } | } \right)^2 } \right)^2 \varepsilon\,| \nabla \varphi^{ \varepsilon } |^2\,d \mathcal{ L }^n d t\\
			&= \int_0^T \int_{ \{ | \nabla \varphi^{ \varepsilon } ( \cdot , t ) | > 0 \} } \varepsilon\,| \nabla \varphi^{ \varepsilon } |^2 + \varepsilon\,( \partial_t \varphi^{ \varepsilon } )^2 \, d \mathcal{ L }^n d t
            \leq \sigma ( 2 T + 1 ) \mu_0^{ \varepsilon } ( \R^n ).
            \end{split}
		\end{equation}
		Due to Lemma \ref{discripancy}, $ \varepsilon\,| \nabla \varphi^{ \varepsilon } |^2\,d\mathcal{L}^n dt$ converges to $ \mu $. Hence, according to the compactness theorem of the measure-function pairs (Theorem \ref{measurefunctiontheorem}), there exists a $ \mu $-measurable function $ f\geq 0 $ such that
		\begin{equation}
			\lim_{ \varepsilon \to +0 } \frac{ 1 }{ \sigma } \int_{ \{ | \nabla \varphi^{ \varepsilon } | > 0 \} } \phi \sqrt{ 1 + \left( \frac{ \partial_t \varphi^{ \varepsilon } }{ | \nabla \varphi^{ \varepsilon } | } \right)^2 } \varepsilon\,| \nabla \varphi^{ \varepsilon } |^2\ d \mathcal{L}^n dt = \int_{ \R^n \times \R^+ } \phi\,f\ d \mu \label{measurefunction}
		\end{equation}
        for all $ \phi \in C_c^0 ( \R^n \times ( 0 , \infty ) ) $. On the other hand, we have
        \[
            \left( \sqrt{ \frac{ 2 W ( \varphi^{ \varepsilon } ) }{ \varepsilon } } - \sqrt{ \varepsilon }\,| \nabla \varphi^{ \varepsilon } | \right)^2 \leq 2\,\left| \frac{ \varepsilon\,| \nabla \varphi^{ \varepsilon } |^2 }{ 2 } - \frac{ W ( \varphi^{ \varepsilon } ) }{ \varepsilon } \right|.
        \]
        Therefore we have
		\begin{equation}
			\begin{split}
				&\left| \int_{ \{ | \nabla \varphi^{ \varepsilon } | > 0 \} } \phi\,| \nabla^{ \prime } w ( \varphi^{ \varepsilon } ) | - \int_{ \{ | \nabla \varphi^{ \varepsilon } | > 0 \} } \phi \sqrt{ 1 + \left( \frac{ \partial_t \varphi^{ \varepsilon } }{ | \nabla \varphi^{ \varepsilon } | } \right)^2 } \varepsilon\,| \nabla \varphi^{ \varepsilon } |^2\,\right|\\
				&=\left| \int_{ \{ | \nabla \varphi^{ \varepsilon } | > 0 \} } \phi \sqrt{ 1 + \left( \frac{ \partial_t \varphi^{ \varepsilon } }{ | \nabla \varphi^{ \varepsilon } | } \right)^2 } \left( \sqrt{ \frac{ 2 W ( \varphi^{ \varepsilon } ) }{ \varepsilon } } - \sqrt{ \varepsilon } | \nabla \varphi^{ \varepsilon } | \right) \sqrt{ \varepsilon } \,| \nabla \varphi^{ \varepsilon } |\,\right|\\
				&\leq \sqrt{ 2 } \left( \int_{ \{ | \nabla \varphi^{ \varepsilon } | > 0 \} } \phi^2 \left( 1 + \left( \frac{ \partial_t \varphi^{ \varepsilon } }{ | \nabla \varphi^{ \varepsilon } | } \right)^2 \right) \varepsilon\,| \nabla \varphi^{ \varepsilon } |^2\,\right)^{ \frac{ 1 }{ 2 } } \left( \int_{ \{ | \nabla \varphi^{ \varepsilon } | > 0 \} } \left| \frac{ \varepsilon | \nabla \varphi^{ \varepsilon } |^2 }{ 2 } - \frac{ W ( \varphi^{ \varepsilon } ) }{ \varepsilon } \right|\,\right)^{ \frac{ 1 }{ 2 } },
			\end{split}
			\label{compute}
		\end{equation}
		where we use the H\"{o}lder inequality. Thanks to (\ref{nablaprime})-(\ref{compute}) and Lemma \ref{discripancy}, we conclude that
		\begin{equation}
			\lim_{ \varepsilon \to +0 } \frac{ 1 }{ \sigma } \int_{ \{ | \nabla \varphi^{ \varepsilon } | > 0 \} } \phi\,| \nabla^{ \prime } w ( \varphi^{ \varepsilon } ) |\,d \mathcal{ L }^n d t = \int_{ \R^n \times \R^+ } \phi\,f\,d \mu. \label{abstconti}
		\end{equation}
		On the other hand, by using the H\"{o}lder inequality, we have
		\begin{equation}
			\begin{split}
				\int_{ \{ | \nabla \varphi^{ \varepsilon } | = 0 \} } | \nabla^{ \prime } w ( \varphi^{ \varepsilon } ) |
				&= \int_{ \{ | \nabla \varphi^{ \varepsilon } | = 0 \} } | \partial_t \varphi^{ \varepsilon } | \sqrt{ 2 W ( \varphi^{ \varepsilon } ) }\\
				&\leq \sqrt{ 2 } \left( \int_{ \R^n \times \R^+ } \varepsilon\,( \partial_t \varphi^{ \varepsilon } )^2 \right)^{ \frac{ 1 }{ 2 } } \left( \int_{ \{ | \nabla \varphi^{ \varepsilon } | = 0 \} } \frac{ W ( \varphi^{ \varepsilon } ) }{ \varepsilon } \right)^{ \frac{ 1 }{ 2 } }\\
				&\leq \sqrt{ 2 } \left( \int_{ \R^n \times \R^+ } \varepsilon\,( \partial_t \varphi^{ \varepsilon } )^2 \right)^{ \frac{ 1 }{ 2 } } \left( \int_{ \R^n \times \R^+ } \left| \frac{ \varepsilon\,| \nabla \varphi^{ \varepsilon } |^2 }{ 2 } -\frac{ W ( \varphi^{ \varepsilon } ) }{ \varepsilon } \right|   \right)^{ \frac{ 1 }{ 2 } }.
			\end{split}
		\end{equation}
		Letting $ \varepsilon \to 0 $, again by Lemma \ref{apriori} and Lemma \ref{discripancy}, we have $ \int_{ B^{ \varepsilon } } | \nabla^{ \prime } w ( \varphi^{ \varepsilon } ) |\,d\mathcal{ L }^n d t \to 0 $. Therefore, in general, the equality (\ref{abstconti}) is valid even if we replace $ A^{ \varepsilon } $ by $ \R^n \times ( 0 , \infty ) $. Thus, by \cite[Proposion 1]{modica1987gradient}, the lower semi-continuity of total variation measure and (\ref{abstconti}), we obtain
		\begin{equation}
			\begin{split}
				&\sigma \int_{ \R^n \times \R^+ } \phi\,d | \nabla^{ \prime } \chi_E |
				=\int_{ \R^n \times \R^+ } \phi\,d | \nabla^{ \prime } w ( \varphi ) |\,d \mathcal{ L }^n d t\\
				&\leq \liminf_{ \varepsilon \to +0 } \int_{ \R^n \times \R^+ } \phi\,| \nabla^{ \prime } w ( \varphi^{ \varepsilon } ) |\,d\mathcal{ L }^n dt = \sigma\int_{ \R^n \times \R^+ } \phi\,f\,d \mu
			\end{split}
		\end{equation}
        for $ \phi \geq 0 $. This completes the proof of $ | \nabla^{ \prime } \chi_E | \ll \mu $.
	\end{proof}

	\subsection{Basic Properties of $ L^2 $ Flow and Set of Finite Perimeter}
	\label{Property}
	In this subsection, we state the following properties of $ L^2 $ flow and set of finite perimeter. The proof of Theorem \ref{mainresult1} will follow from those properties. The arguments in this subsection are mostly contained in \cite{roger2008allen,StuvardTonegawa+2022,tashiro2023existence}
	and we include this for the convenience of the reader.
	\begin{prop}
		Let $ \{ \mu_t \}_{ t \in \R^+ } $ and $ \{ E_t \}_{ t \in \R^+ } $ be as in Theorem \ref{Brakketrans}. We set $ d \mu = d \mu_t d t $ and $E$ as in \eqref{defE}. Then $ \mu \llcorner_{ \partial^* E } $ is an $n$-dimensional rectifiable Radon measure and we have the following for $ \mathcal{ H }^n $-a.e.\,$ ( x , t ) \in \partial^* E \cap \{ t > 0 \} $:
		\begin{description}
			\item[\textup{(1)}] the tangent space $ T_{ ( x , t ) }\, \mu $ exists, and $ T_{ ( x , t ) } \,\mu = T_{ ( x , t ) } ( \partial^* E ) $,
			\item[\textup{(2)}] $ \begin{pmatrix}
				h ( x , t )\\
				1
			\end{pmatrix}
			\in T_{ ( x , t ) }\, \mu $,
			\item[\textup{(3)}] $ x \in \partial^* E_t $, and $ T_x \, \mu_t = T_x \, ( \partial^* E_t ) $,
			\item[\textup{(4)}] $ \mathbf{ p } ( \nu_E ( x , t ) ) \neq 0 $, and $ \nu_{ E_t } ( x ) = | \mathbf{ p } ( \nu_E ( x , t ) ) |^{ - 1 }\,\mathbf{ p } ( \nu_E ( x , t ) ) $,
			\item[\textup{(5)}] $ T_x\,( \partial^* E_t ) \times \{ 0 \} $ is linear subspace of $ T_{ ( x , t ) }\, \mu $.
		\end{description}
		\label{princi}	
	\end{prop}
	The key step of the proof of Theorem \ref{mainresult1} is to prove the above proposition, for which the property of $ L^2 $ flow plays a central role, and this proposition is proved in detail by \cite[Lemma 4.7]{StuvardTonegawa+2022}. In this paper, we will give a brief outline of the proof of Proposition \ref{princi}.
	\vskip.5\baselineskip
	First, we introduce the following property for Brakke flows \cite[Theorem 4.3]{StuvardTonegawa+2022}. Note that the following claim holds for the generalized Brakke flows of \cite{takasao2016existence} and \cite{jiang2020convergence} with a slight modification of the proof.
	\begin{prop}
		\label{BtoL2}
		Let $ \{ \mu_t \}_{ t \in \R^+ } $ be a Brakke flow as in Definition \ref{Brakke}. Then $ \{ \mu_t \}_{ t \in \R^+ } $ is an $ L^2 $ flow with the velocity $ v = h $. Namely, there exists $ C = C ( \mu ) > 0 $ such that
		\begin{equation}
			\label{L2ineq}
			\left| \int_{ \R^+ } \int_{ \R^n } \partial_t \phi + \nabla \phi \cdot h\ d \mu_t d t\,\right| \leq C\,\| \phi \|_{ C^0}, 
		\end{equation}
		for all $ \phi \in C^1_c ( \R^n \times ( 0 , \infty) ) $.
	\end{prop}
 
	The following is a simple property of $ L^2 $ flow (\cite[Proposition 3.3]{roger2008allen}).
	\begin{prop}
		Let $ \{ \mu_t \}_{ t \in \R^+ } $ be an $ L^2 $ flow with the velocity $ v $ as in Definition \ref{L2def}. Let $ \mu $ be the space-time measure $ d \mu = d \mu_t d t $. Then,
		\begin{equation}
			\begin{pmatrix}
				v( x , t )\\
				1
			\end{pmatrix} \in T_{ ( x , t ) }\, \mu
		\end{equation}
		at $ \mu $-a.e.\,$ ( x , t ) $ wherever the tangent space $ T_{ ( x , t ) }\,\mu $ exists.
		\label{L2cor}
	\end{prop}

	For the proof of Proposition \ref{princi}, we need the following general facts on sets of finite perimeter (\cite[Theorem 18.11]{maggi2012sets}).
	\begin{lemma}
		If $ E \subset \R^n \times \R $ is a set of locally finite perimeter, then the horizontal section $ E_t $ is a set of locally finite perimeter in $ \R^n $ for a.e.\,$ t \in \R $, and the following properties hold:
		\begin{description}
			\item[\textup{(1)}] $ \mathcal{ H }^{ n - 1 } ( \partial^* E_t \Delta ( \partial^* E )_t ) = 0 $,
			\item[\textup{(2)}] $ \mathbf{ p } ( \nu_E ( x , t ) ) \neq 0 $ for $ \mathcal{ H }^{ n - 1 } $-a.e.\,$ x \in ( \partial^* E )_t $,
			\item[\textup{(3)}] $ \nabla \chi_{ E_t } = | \mathbf{ p } ( \nu_E ( x , t ) ) |^{ - 1 } \mathbf{ p } ( \nu_E ( x , t ) )\,\mathcal{ H }^{ n - 1 } \llcorner_{ ( \partial^* E )_t } $.
		\end{description}
		\label{Slice}
	\end{lemma}
	\begin{proof}[Proof of Proposition \ref{princi}]
		First of all, we will prove that $ \mu \llcorner_{\partial^* E} $ is a rectifiable Radon measure. It is not difficult to see that $ \mu \ll \mathcal{H}^n $. Indeed, let $ A \subset \R^n \times \R $ be a set with $ \mathcal{H}^n ( A ) = 0 $, and let the set $ D_k := \{ ( x , t ) \in \R^n \times \R^+ \mid \Theta^{*n} ( \mu , ( x , t ) ) \leq k \} $ for each $ k \in \N $. By \cite[Theorem 3.2]{simon1983lectures}, we have
		\[
		\mu ( A \cap D_k ) \leq 2^n\,k\,\mathcal{H}^n( A \cap D_k ) = 0
		\]
		for all $ k \in \N $. Furthermore, by the upper bound of mass density ratio (Lemma \ref{denslemma}), we see that $ \mu ( A \backslash \bigcup_{k=1}^{\infty} D_k ) = 0 $. Thus we obtain $ \mu ( A ) = 0 $, that is, $ \mu \ll \mathcal{H}^n $ holds. Since $ \mu \ll \mathcal{H}^n $, $ | \nabla^{\prime } \chi_E | = \mathcal{H}^n \llcorner_{\partial^* E} $ and Proposition \ref{varphi<<mu}, we see that
		\[
		\mu \llcorner_{\partial^* E} \ll | \nabla^{\prime} \chi_E | , \quad | \nabla^{\prime} \chi_E | \ll \mu \llcorner_{\partial^* E}.
		\]
		By Radon--Nikod\'{y}m theorem, there exists an $ L^1_{ loc } ( | \nabla' \chi_E | ) $ function $ f =( d \mu \llcorner_{ \partial^* E } ) / d | \nabla' \chi_E | $ with $ 0 < f < \infty $ for $ | \nabla' \chi_E | $-a.e.\,and $ \mu \llcorner_{ \partial^* E } = f\,| \nabla' \chi_E | = f\,\mathcal H^n \llcorner_{ \partial^* E }$. This shows that $ \mu \llcorner_{ \partial^* E } $ is a rectifiable Radon measure and the tangent space $ T_{ ( x , t ) }( \mu \llcorner_{ \partial^* E } ) $ with multiplicity $ f $ exists for $ \mathcal H^n $-a.e.\,$ ( x , t ) \in \partial^* E \cap \{ t > 0 \} $. For the next step, we prove that $ T_{ ( x , t ) }\,\mu = T_{ ( x , t ) }\,( \partial^* E ) $ for $ \mathcal{ H }^n $-a.e.\,$ ( x, t ) \in \partial^* E \cap \{ t > 0 \} $. Now, by \cite[Theorem 3.5]{simon1983lectures}, we see that
		\[
		\limsup_{ r \to + 0 } \frac{ \mu ( B^{ n + 1 }_r ( x , t ) ) \setminus \partial^* E ) }{ r^n } = 0 \quad \text{for $ \mathcal{ H }^n $-a.e.\,} ( x , t ) \in \partial^* E \cap \{ t > 0 \}.
		\]
		Let then $ \phi \in C^0_c ( B^{ n + 1 }_1 ( 0 ) ) $ be arbitrary, we have
		\begin{equation*}
			\begin{split}
				\lim_{ r \to + 0 } &\left| \int_{ \R^n \times ( 0 , \infty ) \setminus \partial^* E } \frac{ 1 }{ r^n } \, \phi \left( \frac{ 1 }{ r } ( y - x , s - t ) \right) d \mu ( y , s ) \, \right|\\
				&\leq \| \phi \|_{ C^0 } \limsup_{ r \to + 0 } \frac{ \mu ( B^{ n + 1 }_r ( x , t ) ) \setminus \partial^* E ) }{ r^n } = 0
			\end{split}
		\end{equation*}
		for $ \mathcal{ H }^n $-a.e.\,$ ( x, t ) \in \partial^* E \cap \{ t > 0 \} $. Thus, by $ f \in L^1_{ loc } ( | \nabla' \chi_E | ) $, we obtain at each Lebesgue point of $ f $
		\begin{equation*}
			\begin{split}
				\lim_{ r \to + 0 } \int_{ \R^n \times ( 0 , \infty ) } \frac{ 1 }{ r^n }\, \phi &\left( \frac{ 1 }{ r } ( y - x , s - t ) \right) d \mu ( y , s )\\
				&= \lim_{ r \to + 0 } \int_{ \partial^* E } \frac{ 1 }{ r^n }\,\phi \left( \frac{ 1 }{ r } ( y - x , s - t ) \right)\,\frac{ d \mu }{ d | \nabla^{\prime} \chi_E | } ( y , s )\,d \mathcal{ H }^n ( y , s )\\
				&= f ( x , t ) \int_{ T_{ ( x , t ) } ( \partial^* E ) } \phi ( y , s )\,d \mathcal{ H }^n ( y , s )
			\end{split}
		\end{equation*}
		for all $ \phi \in C^0_c (\R^n\times\R) $ and $ \mathcal{ H }^n $-a.e.\,$ ( x, t ) \in \partial^* E \cap \{ t > 0 \} $. This completes the proof of $ T_{ ( x , t ) }\,\mu = T_{ ( x , t ) }( \partial^* E ) $.
		\vskip.5\baselineskip
		By Proposition \ref{BtoL2}, Proposition \ref{L2cor}, and the above argument, (1) and (2) are proved. Next, we prove (3) and (4). By Lemma \ref{Slice}, we see that the following for a.e.\,$ t > 0 $ and $ \mathcal{H}^{n-1} $-a.e.\,$ x \in (\partial^* E)_t $:
		\begin{align}
			&\mathcal{ H }^{ n - 1 } ( \partial^* E_t \Delta ( \partial^* E )_t ) = 0, \label{Sliceprop1}\\
			&\mathbf{ p } ( \nu_E ( x , t ) ) \neq 0 ,\label{Sliceprop2}\\
			&\nu_{ E_t } ( x ) = \frac{ \mathbf{ p } ( \nu_E ( x , t ) ) }{ | \mathbf{ p } ( \nu_E ( x , t ) ) | }. \label{Sliceprop3}
		\end{align}
		Let $ A := \{ t > 0 \mid \text{ (\ref{Sliceprop1}) fails} \} $ and set $ A_t := \{ x \in (\partial^* E)_t \mid x \notin \partial^* E_t \text{ or (\ref{Sliceprop2})-(\ref{Sliceprop3}) fail} \} $ for every $ t > 0 $, so that $ \mathcal{ L }^1 ( A ) = 0 $ and $ \mathcal{ H }^{ n - 1 } ( A_t ) = 0 $ for every $ t \in ( 0 , \infty ) \setminus A $. Consider then the characteristic function $ \chi ( x , t ) := \chi_{ A_t } ( x ) $ on $ \R^n \times ( 0 , \infty ) $, since $ \mathcal{ L }^1 ( A ) = 0 $ and $ \mathcal{ H }^{ n - 1 } ( A_t ) = 0 $ for every $ t \in ( 0 , \infty ) \setminus A $, we have
		\begin{equation*}
			\begin{split}
				\int_{ \partial^* E } \chi ( x , t )\, | \nabla^{ \partial^* E } ( \mathbf{ q } ( x , t ) ) | \, d \mathcal{ H }^n ( x , t )
				&= \int_0^{ \infty } \int_{ ( \partial^* E )_t} \chi ( x , t ) \, d \mathcal{ H }^{ n - 1 } d t\\
				&= \int_0^{ \infty } \mathcal{ H }^{ n - 1 } ( A_t )\,d t = \int_A \mathcal{ H }^{ n - 1 } ( A_t )\,d t = 0,
			\end{split}
		\end{equation*}
		where we used the co-area formula in the first line, and where $ \nabla^{ \partial^* E } $ is the gradient on the tangent plane of $ \partial^* E $, that is,
		\[
		\nabla^{ \partial^* E } \mathbf{ q } ( x , t ) = P_{ T_{ ( x , t ) } ( \partial^* E ) } ( \nabla \mathbf{ q }  ( x , t ) ).
		\]
		Here, combining (1) and (2), we see that
		\[
		\begin{pmatrix}
			h ( x , t )\\
			1
		\end{pmatrix} \in T_{ ( x , t ) } ( \partial^* E ) \quad \text{at $ \mathcal{ H }^n $-a.e.\,$ ( x , t ) \in \partial^* E \cap \{ t > 0 \} $},
		\]
		which implies $ | \nabla^{ \partial^* E } ( \mathbf{ q } ( x , t ) ) | > 0 $ for $ \mathcal{ H }^n $-a.e.\,$ ( x , t ) \in \partial^* E \cap \{ t > 0 \} $. Hence, it must be $ \chi ( x , t ) = 0 $ for $ \mathcal{ H }^n $-a.e.\,$ ( x , t ) \in \partial^* E \cap \{ t > 0 \} $, thus the first part of (3) and (4) is proved. For the proof of the identity $ T_x\,\mu_t = T_x\,( \partial^* E_t ) $, it is obtained by Lemma \ref{varphi<<mu}, Definition \ref{Brakke} (2) and repeating the argument of (1) at fixed $ t $.
		\vskip.5\baselineskip
		Finally, we prove (5). Taking the $ ( x , t ) \in \partial^* E $ as satisfying (1)-(4) of this Proposition, we can calculate as
		\[
		{}^t ( z , 0 ) \cdot \nu_E ( x , t )
		= z \cdot \mathbf{ p } ( \nu_E ( x , t ) )
		= | \mathbf{ p } ( \nu_E ( x , t ) ) |\,( z \cdot \nu_{ E_t } ( x ) )
		= 0
		\]
		for all $ z \in T_x\,( \partial^* E_t ) $. This completes the proof of (5).
	\end{proof}
	
	\subsection{Boundaries move by mean curvature}
	In this subsection, we prove Theorem \ref{mainresult1} by rephrasing the velocity $ v $ as the mean curvature, and by using geometric measure theory. The argument for this rephrasing corresponds to the proof of the area formula (\ref{theareachange}).
	\vskip.5\baselineskip
	\begin{proof}[Proof of Theorem \ref{mainresult1}]
        Let fix a test function $ \phi \in C^1_c ( \R^n \times ( 0 , \infty ) ) $ arbitrarily. Then, by using Gauss-Green's theorem for set of finite perimeter, we have
		\begin{equation}
		  \int_{ \R^n \times ( 0 , \infty ) } \partial_t \phi\, \chi_E\,d x d t = \int_{ \partial^* E } \phi\,\mathbf{ q } ( \nu_{ E } )\ d\mathcal{ H }^n. \label{Dt1S}
		\end{equation}
		Let $ G $ be the set satisfying Proposition \ref{princi} (1)-(5). Then for all $ ( x , t ) \in G $, we have
		\begin{equation}
			\label{span}
			T_{ ( x , t ) }\, \mu = ( T_x\,( \partial^* E_t ) \times \{ 0 \} ) \oplus \mathrm{ span }
			\begin{pmatrix}
				h ( x , t )\\
				1
			\end{pmatrix} \quad (\text{by Proposition \ref{princi} (2)}),
		\end{equation}
		By $ h ( x , t ) \perp T_x \, \mu_t $, (\ref{span}), Proposition \ref{princi} (1) and (4), we have
		\begin{equation}
			\nu_{ E } ( x , t ) = \frac{ 1 }{ \sqrt{ 1 + | h ( x , t ) |^2 } }
			\begin{pmatrix}
				\nu_{ E_t } ( x )\\
				- h ( x , t ) \cdot \nu_{ E_t } ( x )
			\end{pmatrix}.
                \label{nu-t-component}
		\end{equation}
		By (\ref{nu-t-component}) and $ h ( x , t ) \perp T_x\,\mu_t $ again, for all $ ( x , t ) \in G $, we can calculate the $ i \times ( n + 1 ) $ component of the matrix $ I_{ n + 1 } - \nu_E \otimes \nu_E $ for $ i = 1, \ldots , n 
        + 1 $ as
        \[
		( I_{ n + 1 } - \nu_{ E } \otimes \nu_{ E } )_{ i ,\, ( n + 1 ) } ( x , t ) =
		  \begin{cases}
			\ \frac{ - ( \nu_{ E_t } ( x ) )_i ( h ( x , t ) \cdot \nu_{ E_t } ( x ) ) }{ 1 + | h ( x , t ) |^2 } &( i = 1 , \ldots , n ),\\
			\ \frac{ 1 }{ 1 + | h ( x , t ) |^2 } &( i = n + 1 ),
		\end{cases}
	\]
	where $ I_{ n + 1 } $ is the $ ( n + 1 ) $-identity matrix and $ ( \nu_{ E_t } )_i $ is the $ i $-th component of $ \nu_{ E_t } $
        According to this calculation, $ \nabla {\bf q } = \mathbf{ e }_{ n + 1 } $ and $ T_{ ( x , t ) }\,\mu = T_{ ( x , t ) }\,( \partial^* E ) $ on $ G $, we obtain that the co-area factor of the projection ${\bf q}$ satisfies
		\begin{equation}
			| \nabla^{ \partial^* E } \mathbf{ q } ( \nu_E ( x , t ) ) | = \frac{ 1 }{ \sqrt{ 1 + | h ( x , t ) |^2 } },
			\label{areaelement}
		\end{equation}
		Due to (\ref{Dt1S})-(\ref{areaelement}) and the co-area formula, we compute as
		\begin{equation}
			\begin{split}
				&\int_{ \R^n \times ( 0 , \infty ) } \partial_t \phi\,\chi_E\,dxdt
				= - \int_{ G } \phi\,h \cdot \nu_{ E_t } \,\frac{ 1 }{ \sqrt{ 1 + | h |^2 } }\,d\mathcal{ H }^n
				= - \int_{ \partial^* E } \phi\,h \cdot \nu_{ E_t }\,| \nabla^{ \partial^* E } \mathbf{ q } ( \nu_E ) |\,d\mathcal{ H }^n\\
				&= - \int_{ 0 }^{ \infty } \int_{ \partial^* E \cap \{ \mathbf{ q } = t \} } \phi\,h \cdot \nu_{ E_t }\,d\mathcal{ H }^{ n - 1 } dt
				= - \int_{ 0 }^{ \infty } \int_{ \partial^* E_t } \phi\,h \cdot \nu_{ E_t }\,d\mathcal{ H }^{ n - 1 } dt, \label{GBV}
			\end{split}
		\end{equation}
		where we used $ \mathcal{ H }^n ( \partial^* E \setminus G ) = 0 $.
		By a suitable approximation of $\phi$ in (\ref{GBV}), we deduce
		\begin{equation}
				\int_{ E_{ t_2 } } \phi ( x , t_2 )\,dx - \int_{ E_{ t_1 } } \phi ( x , t_1 )\,dx
				= \int_{ t_1 }^{ t_2 } \int_{ E_t } \partial_t \phi\,dx dt + \int_{ t_1 }^{ t_2 } \int_{ \partial^* E_t } \phi\,h \cdot \nu_{ E_t }\,d\mathcal{ H }^{ n - 1 } dt
		\end{equation}
		for a.e.~$ 0 < t_1 < t_2 < \infty $. Using the continuity of $ \| \chi_{ E_t } \|_{ L^1 ( \R^n ) } $ again, we obtain the above equality for all $ 0 \leq t_1 < t_2 < \infty $ and all $ \phi \in C^1_c ( \R^n \times \R^+ ) $. This completes the proof.
	\end{proof}
	
	\begin{remark}
		In \cite{takasao2023existence} and \cite{takasao2020existence}, the existence theorem for volume-preserving MCFs and MCFs with transport and forcing term was proved in the $ L^2 $ flow sense, not in the Brakke's sense. The argument of this paper to prove (\ref{theareachange}) can also be applied to \cite{takasao2023existence} and \cite{takasao2020existence}, since the property of Brakke flow is only used for the weak solution of the MCF to become an $ L^2 $ flow.
	\end{remark}
	
		\subsection*{Acknowledgment}
	The author would like to thank his supervisor Yoshihiro Tonegawa for his insightful feedback, careful reading, and improving the quality of this paper. The author was supported by JST, the establishment of university fellowships towards the creation of science technology innovation, Grant Number JPMJFS2112.

	\appendix
	\section{Measure-Function Pairs}
	Here, we recall the notion of measure-function pairs introduced by Hutchinson in \cite{hutchinson1986second}.
	\begin{definition}
		Let $ E \subset \R^n $ be an open set and let $ \mu $ be a Radon measure on $ E $. Suppose $ f \in L^1 ( \mu ; \R^d ) $. Then we say that $ ( \mu , f ) $ is $ \R^d $-valued measure-function pair over $ E $.
	\end{definition}
	Next, we define the notion of convergence for a sequence of $ \R^d $-value measure-function pairs over $ E $.
	\begin{definition}
		Let $ \{ ( \mu_i , f_i ) \}_{ i = 1 }^{ \infty } $ and $ ( \mu , f ) $ be $ \R^d $-valued measure-function pairs over $ E $. Suppose
		\[
		\mu_i \rightharpoonup \mu
		\] 
		as Radon measure on $ E $. Then we call $ ( \mu_i , f_i ) $ converges to $ ( \mu , f ) $ in the weak sense if
		\[
		\int_E f_i \cdot \phi\,d \mu_i \to \int_E f \cdot \phi\,d \mu
		\]
		for all $ \phi \in C^0_c ( E ; \R^d ) $.
	\end{definition}
	We present a less general version of \cite[Theorem 4.4.2]{hutchinson1986second} to the extent that it can be used in this paper.
	\begin{theorem}
		\label{measurefunctiontheorem}
		Suppose that $ \R^d $-valued measure-function pairs $ \{ ( \mu_i , f_i ) \}_{ i = 1 }^{ \infty } $ are satisfied
		\[
		\sup_i \int_E | f_i |^2\,d \mu_i < \infty.
		\]
		Then the following hold:
		\begin{description}
			\item[\textup{(1)}] There exist a subsequence $ \{ ( \mu_{ i_j } , f_{ i_j } ) \}_{ j = 1 }^{ \infty } $ and $ \R^d $-value measure-function pair $ ( \mu , f ) $ such that $ ( \mu_{ i_j } , f_{ i_j } ) $ converges to $ ( \mu , f ) $ as measure-function pair.
			\item[\textup{(2)}] If $ ( \mu_{ i_j } , f_{ i_j } ) $ converges to $ ( \mu , f ) $ then
			\[
			\int_E | f |^2\,d \mu \leq \liminf_{ j \to \infty } \int_E | f_{ i_j } |^2\,d \mu_{ i_j } < \infty.
			\]
		\end{description}
	\end{theorem}

	\bibliography{myref.bib}

\begin{thebibliography}{10}

\bibitem{brakke1978motion}
K.~A. Brakke.
\newblock {\em The motion of a surface by its mean curvature}, volume~20 of
  {\em Math. Notes (Princeton)}.
\newblock Princeton University Press, Princeton, 1978.

\bibitem{chen1991uniqueness}
Y.~G. Chen, Y.~Giga, and S.~Goto.
\newblock Uniqueness and existence of viscosity solutions of generalized mean
  curvature flow equations.
\newblock {\em J. Differ. Geom.}, 33(3):749--786, (1991).

\bibitem{evans1991motion}
L.~C. Evans and J.~Spruck.
\newblock Motion of level sets by mean curvature. {I}.
\newblock {\em J. Differ. Geom.}, 33(3):635--681, (1991).

\bibitem{fischer2020local}
J.~Fischer, S.~Hensel, T.~Laux, and T.~Simon.
\newblock The local structure of the energy landscape in multiphase mean
  curvature flow: Weak-strong uniqueness and stability of evolutions.
\newblock {\em arXiv preprint arXiv:2003.05478}, (2021).
\newblock First part accepted at {\itshape J. Eur. Math. Soc.}

\bibitem{fischer2020convergence}
J.~Fischer, T.~Laux, and T.~M. Simon.
\newblock Convergence rates of the {A}llen--{C}ahn equation to mean curvature
  flow: A short proof based on relative entropies.
\newblock {\em SIAM J. Math. Anal.}, 52(6):6222--6233, (2020).

\bibitem{hensel2021new}
S.~Hensel and T.~Laux.
\newblock A new varifold solution concept for mean curvature flow: Convergence
  of the {A}llen-{C}ahn equation and weak-strong uniqueness.
\newblock {\em arXiv preprint, arXiv:2109.04233}, (2021).
\newblock To appear in {\itshape J. Differ. Geom.}

\bibitem{hensel2023sharp}
S.~Hensel and Y.~Liu.
\newblock The sharp interface limit of a {N}avier--{S}tokes/{A}llen--{C}ahn
  system with constant mobility: Convergence rates by a relative energy
  approach.
\newblock {\em SIAM J. Math. Anal.}, 55(5):4751--4787, (2023).

\bibitem{huisken1990}
G.~Huisken.
\newblock Asymptotic behavior for singularities of the mean curvature flow.
\newblock {\em J. Differ. Geom.}, 31(1):285--299, (1990).

\bibitem{hutchinson1986second}
J.~E. Hutchinson.
\newblock Second fundamental form for varifolds and the existence of surfaces
  minimising curvature.
\newblock {\em Indiana Univ. Math. J.}, 35(1):45--71, (1986).

\bibitem{ilmanen1993convergence}
T.~Ilmanen.
\newblock Convergence of the {A}llen-{C}ahn equation to {B}rakke's motion by
  mean curvature.
\newblock {\em J. Differ. Geom.}, 38(2):417--461, (1993).

\bibitem{ilmanen1994elliptic}
T.~Ilmanen.
\newblock Elliptic regularization and partial regularity for motion by mean
  curvature.
\newblock {\em Mem. Am. Math. Soc.}, 108(520), (1994).

\bibitem{jiang2020convergence}
G.-C. Jiang, C.-J. Wang, and G.-F. Zheng.
\newblock Convergence of solutions of some {A}llen-{C}ahn equations to
  {B}rakke’s mean curvature flow.
\newblock {\em Acta Appl. Math.}, 167(1):149--169, (2020).

\bibitem{kasai2014general}
K.~Kasai and Y.~Tonegawa.
\newblock A general regularity theory for weak mean curvature flow.
\newblock {\em Cal. Var. Partial Differ. Equ.}, 50(1):1--68, (2014).

\bibitem{kroemer2023quantitative}
M.~Kroemer and T.~Laux.
\newblock Quantitative convergence of the nonlocal {A}llen--{C}ahn equation to
  volume-preserving mean curvature flow.
\newblock {\em arXiv preprint arXiv:2309.12409}, (2023).

\bibitem{laux2016convergence}
T.~Laux and F.~Otto.
\newblock Convergence of the thresholding scheme for multi-phase mean-curvature
  flow.
\newblock {\em Calc. Var. Partial Differ. Equ.}, 55(5):1--74, (2016).

\bibitem{laux2018convergence}
T.~Laux and T.~M. Simon.
\newblock Convergence of the {A}llen-{C}ahn equation to multiphase mean
  curvature flow.
\newblock {\em Commun. Pure Appl. Math.}, 71(8):1597--1647, (2018).

\bibitem{luckhaus1995implicit}
S.~Luckhaus and T.~Sturzenhecker.
\newblock Implicit time discretization for the mean curvature flow equation.
\newblock {\em Calc. Var. Partial Differ. Equ.}, 3(2):253--271, (1995).

\bibitem{maggi2012sets}
F.~Maggi.
\newblock {\em Sets of finite perimeter and geometric variational problems: an
  introduction to Geometric Measure Theory}, volume 135 of {\em Camb. Stud.
  Adv. Math.}
\newblock Cambridge University Press, Cambridge, 2012.

\bibitem{modica1987gradient}
L.~Modica.
\newblock The gradient theory of phase transitions and the minimal interface
  criterion.
\newblock {\em Arch. Ration. Mech. Anal.}, 98(2):123--142, (1987).

\bibitem{roger2008allen}
L.~Mugnai and M.~R{\"o}ger.
\newblock The {A}llen--{C}ahn action functional in higher dimensions.
\newblock {\em Interfaces Free Bound.}, 10(1):45--78, (2008).

\bibitem{qi2018convergence}
Y.~Qi and G.-F. Zheng.
\newblock Convergence of solutions of the weighted {A}llen--{C}ahn equations to
  {B}rakke type flow.
\newblock {\em Calc. Var. Partial Differ. Equ.}, 57(5):133, (2018).

\bibitem{simon1983lectures}
L.~Simon.
\newblock {\em Lectures on geometric measure theory}, volume~3 of {\em Proc.
  Cent. Math. Anal. Aust. Natl. Univ.}
\newblock Australian National University, Canberra, 1983.

\bibitem{stuvard2022endtime}
S.~Stuvard and Y.~Tonegawa.
\newblock End-time regularity theorem for {B}rakke flows.
\newblock {\em Math. Ann.}, (2024).
\newblock Online first, https://doi.org/10.1007/s00208-024-02826-8.

\bibitem{StuvardTonegawa+2022}
S.~Stuvard and Y.~Tonegawa.
\newblock On the existence of canonical multi-phase {B}rakke flows.
\newblock {\em Adv. Calc. Var.}, 17(1):33--78, (2024).

\bibitem{takasao2017existence}
K.~Takasao.
\newblock Existence of weak solution for volume-preserving mean curvature flow
  via phase field method.
\newblock {\em Indiana Univ. Math. J.}, 66(6):2015--2035, (2017).

\bibitem{takasao2020existence}
K.~Takasao.
\newblock Existence of weak solution for mean curvature flow with transport
  term and forcing term.
\newblock {\em Commun. Pure Appl. Anal.}, 19(5):2655--2677, (2020).

\bibitem{takasao2023existence}
K.~Takasao.
\newblock The existence of a weak solution to volume preserving mean curvature
  flow in higher dimensions.
\newblock {\em Arch. Ration. Mech. Anal.}, 247(3):52, (2023).

\bibitem{takasao2016existence}
K.~Takasao and Y.~Tonegawa.
\newblock Existence and regularity of mean curvature flow with transport term
  in higher dimensions.
\newblock {\em Math. Ann.}, 364(3):857--935, (2016).

\bibitem{tashiro2023existence}
K.~Tashiro.
\newblock Existence of bv flow via elliptic regularization.
\newblock {\em arXiv preprint, arXiv:2305.12374,}, (2023).
\newblock To appear in {\itshape Hiroshima Math. J.}

\bibitem{tonegawa2003integrality}
Y.~Tonegawa.
\newblock Integrality of varifolds in the singular limit of reaction-diffusion
  equations.
\newblock {\em Hiroshima Math. J.}, 33(3):323--341, (2003).

\bibitem{tonegawa2014second}
Y.~Tonegawa.
\newblock A second derivative {H}{\"o}lder estimate for weak mean curvature
  flow.
\newblock {\em Adv. Calc. Var.}, 7(1):91--138, (2014).

\end{thebibliography}
	\bibliographystyle{abbrv}
	
	\hrulefill
	
\end{document}